\documentclass[11pt,reqno]{amsart}
\usepackage{graphicx}
\usepackage{amssymb}
\usepackage{amsthm}
\usepackage{dsfont}
\usepackage{hyperref}
\usepackage{mathrsfs}
\usepackage{stmaryrd}
\usepackage{amsrefs}
\usepackage{amsmath}
\usepackage{centernot}
\usepackage{mathtools}
\usepackage{xcolor}

\newtheorem{theorem}{Theorem}[section]
\newtheorem{lemma}[theorem]{Lemma}

\newtheorem{corollary}[theorem]{Corollary}
\numberwithin{equation}{section}
\theoremstyle{definition}
\newtheorem{example}{Example}[section]

\theoremstyle{definition}

\theoremstyle{remark}
\newtheorem{remark}{Remark} 
\allowdisplaybreaks

\newcommand{\IP}{\mathbf{P}}

\newcommand{\con}[1]{ \xleftrightarrow{#1}}
\newcommand{\ncon}[1]{ \centernot{\xleftrightarrow{#1}}}

\newcommand{\set}{\omega}
\newcommand{\setm}{\omega}
\newcommand{\setf}{\omega}
\newcommand{\sets}{\Gamma}
\newcommand{\currs}{\Omega}
\newcommand{\flows}{\mathcal{F}}
\newcommand{\n}{\mathbf{n}}
\newcommand{\rcur}{\textnormal{curr}}
\newcommand{\drcur}{\textnormal{d-curr}}
\newcommand{\aflow}{\textnormal{a-flow}}

\newcommand{\ising}{\textnormal{Isg}}
\newcommand{\para}{\mathcal{P}}

\DeclareMathOperator{\pf}{\textnormal{Pf}}
\newcommand{\ccomp}{\kappa}
\newcommand{\ec}{\tilde F}
\newcommand{\weight}{\textnormal{w}}
\newcommand{\source}{{+}}
\newcommand{\sink}{{-}}

\makeatletter
\newcommand\connect[2][]{%
  \ext@arrow 9999{\longleftrightarrowfill@}{#1}{#2}}
\newcommand\longleftrightarrowfill@{%
  \arrowfill@\leftarrow\relbar\rightarrow}
\makeatother

\title[The planar Ising model and total positivity]{The planar Ising model and total positivity}

\author{Marcin Lis}{
\address{Marcin Lis\\Statistical Laboratory\\ Centre for Mathematical Sciences\\Cambridge University\\
Wilberforce Road\\Cambridge CB3 0WB\\UK}
\email{m.lis@statslab.cam.ac.uk}}
 
\date{\today}


\begin{document}

\begin{abstract}
A matrix is called totally positive (resp.\ totally nonnegative) if all its minors
are positive (resp.\ nonnegative). 
Consider the Ising model with free boundary conditions and no external field on a planar graph $G$. Let $a_1,\dots,a_k,b_k,\dots,b_1$ be vertices placed in a counterclockwise order on the outer 
face of $G$. We show that the $k\times k$ matrix of the two-point spin correlation functions
\[
 M_{i,j} = \langle \sigma_{a_i} \sigma_{b_j} \rangle 
 \]
is totally nonnegative. Moreover, $\det  M > 0$ if and only if there exist $k$ pairwise vertex-disjoint paths that connect $a_i$ with $b_i$.
We also compute the scaling limit at criticality of the probability that there are $k$ parallel and disjoint connections between $a_i$ and $b_i$ in the double random current model.
Our results are based on a new distributional relation between double random currents and random alternating flows of Talaska~\cite{talaska}.
\end{abstract}

\maketitle

\section{Introduction}

The Ising model was introduced by Lenz with the intention to describe the behaviour of ferromagnets, and was first solved in dimension $1$ by Ising~\cite{Ising}.
Peierls later showed that the model does undergo a phase transition in dimensions $2$ or more~\cite{Peierls}, 
and it has been since the subject of extensive study both in the physics and mathematics literature. 
Notable results in the planar case include the exact solution obtained by Onsager~\cite{Onsager} and Yang~\cite{Yang}, 
and the recent breakthrough 
of Smirnov et al.\ showing conformal invariance in the critical scaling limit~\cites{smirnov,CheSmi,CHI,hongler,HonSmi}.

A seemingly unrelated notion is that of totally positive matrices characterized by having all their minors positive.
They appear in various areas of mathematics and physics including oscillations in mechanical systems~\cite{GanKre1,GanKre} (which was the original motivation to study them), 
stochastic processes and statistical mechanics~\cites{KarMcG,Curetal,fomin}, quantum groups~\cites{lusztig0,lusztig1,lusztig2}, and algebraic geometry~\cites{postnikov,Posetal}. 
Some of their fundamental properties include 
the simplicity and positivity of the spectrum obtained by Gantmacher and Krein~\cite{GanKre1}, and the variation diminishing property discovered by Shoenberg~\cite{schoenberg} 
which says that the number of sign changes in a vector does not increase after multiplying by a totally positive matrix.

In this article we identify total positivity in the planar Ising model.
Our results on the Ising boundary two-point correlation functions are analogous to the results of Fomin~\cite{fomin} on the walk matrices of random walks on planar graphs.
Indeed, the walk matrix can be interpreted as the two-point correlation function matrix of another very well known ferromagnetic spin system -- the discrete Gaussian free field.
Furthermore, Fomin provides an interpretation of the determinants
of walk matrices in terms of probabilities of non-intersection 
events involving loop-erased walks of Lawler~\cite{lawler}. 
In our setting the relevant events concern the double random current model.

The random current model is derived from the power series expansion of the Ising model partition function.
It was introduced by Griffiths, Hurst and Sherman~\cite{GHS}, and was later
used by Aizenman et al.\ ~\cite{aizenman,ABF,AF} to obtain a detailed description of the behaviour of the Ising model on $\mathds{Z}^d$.
The model has recently received revived attention from the mathematical community.
The most notable example is the result of Aizenman, Duminil-Copin and Sidoravicius~\cite{ADCS} who studied percolation properties of double random currents to show 
continuity of magnetization for a wide range of Ising models on $\mathds{Z}^d$ (including the $d=3$ case). Their argument was later generalized by Bj\"{o}rnberg to the setting of
quantum Ising models~\cite{bjornberg}.
Also, a new distributional relation between random currents, Bernoulli percolation and the FK-Ising model was discovered by Lupu and Werner~\cite{LupWer}.
One of the main tools used to study the random current model is the switching lemma~\cite{GHS}, and 
our result may be thought of as its planar generalization (see Example~\ref{ex:1}).
For a recent overview of the applications of the random current model, see~\cite{DC}.

Since the work of Groeneveld, Boel and Kasteleyn~\cite{GBK}, the boundary Ising correlation functions have been known to satisfy exact Pfaffian relations.
We provide an alternative proof of this fact using the expansion of the Pfaffian as a sum of determinants (see Lemma~\ref{lem:gbk}).
Our results are hence a refinement of those in~\cite{GBK} in that we relate each of these determinants to an explicit event in the double random current model.

A special property of the (dis-)connection events appearing in our results is that even though they are non-local, i.e., they depend on a macroscopic (in terms of the size of the graph) 
number of edge variables, their probabilities are simple functions of the boundary measurements which are 
defined as expectations of local variables.
This in particular implies that these probabilities do not depend on the structure of the graph as long as the boundary measurements are 
preserved.
This phenomenon also exists in uniform groves and double dimers as was discovered by Kenyon and Wilson~\cites{KenWil1,KenWil2}. Moreover, the authors 
proved that the probability of seeing any type of partition of the boundary vertices induced by cluster connectivities in these models is a homogenous multi-linear polynomial 
in the boundary measurements.

We arrive at the total positivity of the Ising boundary two-point functions by representing them in terms of alternating flows of Talaska
which satisfy total positivity~\cite{talaska}. They, in turn, appear in the combinatorial study of the totally nonnegative Grassmannian initiated by Postnikov~\cite{postnikov}. 
The interpretation of the determinants of the two-point functions in terms of random current probabilities is achieved by a new distributional relation between 
double random currents and random alternating flows.

This article is organized as follows: The next section introduces basic notation and presents the main results. Section~\ref{sec:currents} gives a distributional identity for double random currents, and
Section~\ref{sec:flows} provides a closely related identity for random alternating flows. The relationship between these two, together with the total positivity of alternating flows,
is the basis for our main results, whose proofs are given in Section~\ref{sec:proofs}.

\section{Main results}
Let $G=(V,E)$ be a finite connected planar graph.
For positive coupling constants $J : E \to (0,\infty)$, we consider the Ising model on $G$ with free boundary conditions and no external field, i.e.,\ 
a probability measure on the space of spin configurations $\{ -1,1\}^V$ given by
\[
\IP_{\ising} (\sigma)=\frac{1}{Z_{\ising}}\prod_{\{u,v\}\in E} \exp(J_{\{u,v \}} \sigma_u \sigma_v), \qquad \sigma \in  \{ -1,1\}^V,
\]
where 
\[
Z_{\ising} = \sum_{\sigma \in  \{ -1,1\}^V }\prod_{\{u,v\}\in E} \exp(J_{\{u,v \}} \sigma_u \sigma_v)
\]
is the partition function. Since the coupling constants are positive, the model is ferromagnetic, i.e., it assigns larger probability to configurations 
where more pairs of adjacent spins assume the same value.
For $u,v\in V$, the two-point spin correlation function is defined to be the expectation
 \[
 \langle \sigma_u \sigma_v \rangle = \sum_{\sigma \in  \{ -1,1\}^V}  \sigma_u \sigma_v \IP_{\ising} (\sigma).
 \]
 A classical argument says that in a ferromagnetic spin system, all two-point functions are positive (see Lemma~\ref{lem:currtwopoint}). 
  
Let $W=\{ w_1,w_2,\dots,w_n\}$ be the set of vertices on the outer face of $G$ listed counterclockwise, and 
let $W_{ \circ} \cup W_{\bullet} = W$ be a partition of $W$ into two sets (one of them possibly empty). 
We will refer to $W$ as the boundary vertices and for $A\subseteq W$, we will write $A_{\circ} = A \cap W_{\circ}$ and $A_{\bullet}= A \cap W_{\bullet}$.
For two natural numbers $i,j$, we define $\mathcal{I}_{i,j} $ to be the set of numbers strictly between $i$ and $j$, e.g.\ $\mathcal{I}_{1,1}= \emptyset$, and $\mathcal{I}_{1,3}=\mathcal{I}_{3,1}=\{ 2\}$.
The following result yields positivity of determinants of matrices whose entries are, up to a sign, the boundary two-point spin correlation functions.
 \begin{theorem} \label{thm:positivity} Let
 $A=\{w_{l_1}, w_{l_2},\dots, w_{l_k}\} \subseteq W$ with $l_1< l_2 < \dots < l_k$. Consider the $k\times n$ matrix
 \[
N^A _{i,j}=  (-1)^{s(i,j)}  \langle \sigma_{w_{l_i}} \sigma_{w_j} \rangle, 
 \]
 where  
 \[
 s(i,j)= |\mathcal{I}_{l_i,j} \cap \{l_1,\dots, l_k \}| + \mathds{1}_{\{ w_j \in A_{\circ}, \ j < l_i\}}+ \mathds{1}_{\{ w_j \in A_{\bullet}, \ j > l_i\}}.
 \]
 Then, for any $B \subset W$ with $|B| =k$, 
 \[
 \det  N^{A,B} \geq 0,
 \]
 where $N^{A,B}$ is the $k\times k$ submatrix of $N^A$ with columns indexed by $B$. Moreover, $\det N^{A,B} >0$ 
 if and only if there exist $k$ pairwise edge-disjoint (possibly empty) directed paths in $G$ such that each of them starts in $A$ and ends in~$B$, and
 if two paths meet at a vertex, then their edges alternate in orientation around it (edge oriented towards or away from the vertex).
 Here, we assume that an empty path
 joining $v\in A_{\circ}\cap B_{\circ} $ (resp.\ $v\in A_{\bullet}\cap B_{\bullet} $) with itself is represented by a counterclockwise loop (resp.\ clockwise loop) attached to $v$ within the outer face.
 \end{theorem} 
We note that the choice of $W_{ \circ}$ and $W_{\bullet}$ determines both the sign factors appearing in the definition of $N^A$, and 
the interpretation of the associated determinants in terms of alternating flows (see Lemma~\ref{lem:tposflow}). 
This is inherited from the results of~\cite{talaska} through the construction of Section~\ref{sec:flows} (see Figure~\ref{fig:directedgraph}).

 \begin{example} 
For any distinct $a,b,c \in W$, we have 
\[
1 +\langle \sigma_a \sigma_b \rangle > \langle \sigma_a \sigma_c \rangle + \langle \sigma_b \sigma_c \rangle, \quad \text{ and } \quad 1 +\langle \sigma_a \sigma_b \rangle^2 > \langle \sigma_a \sigma_c \rangle^2 + \langle \sigma_b \sigma_c \rangle^2.
\]
Indeed, assume that $a,b,c$ are ordered counterclockwise, $a,b \in W_{\circ}$, $c \in W_{\bullet}$, and take $A,B=\{a,b,c\}$.
We have
\[
N^{A,B} = 
\begin{bmatrix*}[r]
  \multicolumn{1}{c}{1} & \langle \sigma_a\sigma_b\rangle & \langle \sigma_a\sigma_c\rangle \\
 - \langle \sigma_a\sigma_b\rangle & \multicolumn{1}{c}{1} & -\langle \sigma_b\sigma_c\rangle\\
  \langle \sigma_a\sigma_c\rangle &- \langle \sigma_b\sigma_c\rangle & \multicolumn{1}{c}{1}
 \end{bmatrix*},
\]
and by Theorem~\ref{thm:positivity},
\[
\det N^{A,B} = 1+\langle \sigma_a \sigma_b \rangle^2-  \langle \sigma_a\sigma_c \rangle^2 - \langle \sigma_b \sigma_c \rangle^2 > 0,
\]
which gives the second inequality. The inequality is strict since we can take empty paths connecting the vertices with themselves.
To obtain the first inequality, we use the second one together with the Griffiths inequality 
$\langle \sigma_a \sigma_b \rangle\geq  \langle   \sigma_a\sigma_c \rangle \langle \sigma_c \sigma_b \rangle$ \cite{griffiths}. 
\end{example}

In the special case when $A$ and $B$ form disjoint contiguous sets of boundary vertices, we obtain the following total positivity 
property of the boundary two-point spin correlation functions with no additional signs.
\begin{corollary}[Total positivity] \label{cor:totpos}
Let $A=\{ a_1,a_2,\ldots,a_k \} $, $B=\{ b_1, b_2, \ldots b_k\} $ be contiguous sets of boundary vertices, i.e., such that 
$a_1,\ldots,a_k, b_k, \ldots b_1$
is a counterclockwise order on $A\cup B$. Then, the $k\times k$ matrix 
\[
 M^{A,B}_{i,j} = \langle \sigma_{a_i} \sigma_{b_j} \rangle 
\]
is totally nonnegative. Moreover, $ M^{A,B}$ is totally positive if and only if 
for any $A'\subset A$ and $B'\subset B$ such that $|A'|=|B'|=l$, there exist $l$ pairwise vertex-disjoint paths that connect $A'$ and $B'$.
\end{corollary}

\begin{remark}
One can prove that if $A$ and $B$ are as above, then $\det M^{A,B} = \det N^{A,B}$ by showing that the additional signs in $N^{A,B}$ make the same sign contribution to the determinant as the reverse
order on columns in $M^{A,B}$.
This is evident in the following expansion of these determinants.
Let $A,B$ be any subsets of $W$ of equal cardinality, and let $\pi : A \to B$ be a bijection. 
One can interpret $\pi$ as a pairing of the disjoint union $A\sqcup B$, i.e., a partition of $A\sqcup B$ into pairs $\{ a , \pi(a) \}$.
One can then think of a diagrammatic 
representation of $\pi$ where points representing $A$ and $B$ are placed in the corresponding order on the boundary of a disk, and straight line segments connect the points according to $\pi$. 
Here, for each $v \in  A_{\circ}\cap B_{\circ}$ (resp.\ each $v \in  A_{\bullet}\cap B_{\bullet} $), 
two copies of $v$ are placed on the circle, and the one corresponding to a point in $A$ comes immediately after (resp.\ before) the one corresponding to a point in $B$ in the counterclockwise order.
We define $\textnormal{xing}(\pi)$ to be the number of crossings between the line segments. It is easy to see that when $A$, $B$ are as in Corollary~\ref{cor:totpos}, then 
 $\textnormal{xing}(\pi)$ is the number of inversions of $\pi$ treated as a permutation of the index set $\{ 1, \ldots, k\}$. Hence, by the definition of the determinant,
 \begin{align} \label{eq:nicesign}
 \det M^{A,B} = \sum_{\pi : A \to B} (-1)^{\textnormal{xing}(\pi)} \prod_{\{ a,b\} \in \pi} \langle \sigma_a \sigma_b \rangle.
 \end{align}
Moreover, the choice of signs in the definition of the matrix $N^A$ yields an analogous expansion of the determinant for general choices of $A$ and $B$:
\begin{align} \label{eq:nicesign1}
 \det N^{A,B} &= \sum_{\pi: A \to B} (-1)^{\textnormal{xing}(\pi)} \prod_{\{a,b\} \in \pi}\langle \sigma_{a} \sigma_{b} \rangle.
\end{align}
A verification of \eqref{eq:nicesign1} can be found in~\cite[Proposition 5.2]{postnikov} or~\cite[Proposition 2.12]{talaska}.
\end{remark}

Our second main result provides an interpretation of these determinants in terms of random currents. A current on $G$ is a function 
$\n: E \to \{0,1,2,\dots \}$. 
By $\partial \n$ we denote the set of sources of $\n$, i.e., vertices $v$ such that $  \sum_{u : \ \{u,v\} \in E} \n_{\{u,v\}}$
is odd.
For $A \subseteq V $, we define $\currs_A = \{ \n \mid \partial \n =A \}$.
The weight of a current $\n $ is defined by
\begin{align} \label{eq:currentweight}
\weight(\n)=\prod_{e \in E} \frac{ (J_e)^{\n_e}}{\n_e!},
\end{align}
and the random current probability measure with boundary conditions $A$ is given by
\[
\IP^A_{\rcur}(\n) = \frac{\weight(\n)}{Z^A_{\rcur}}, \quad \n \in \currs_A,
\]
where $Z^A_{\rcur} =\sum_{\n \in \currs_A} \weight(\n)$ is the partition function. 

Note that if $\n_1 \in \currs_{A}$ and $\n_2 \in \currs_{\emptyset}$, then $\n_1+\n_2 \in  \currs_{A}$.
The double random current probability measure $\IP^{A}_{\drcur}$ with boundary conditions 
$A$ is defined to be the measure of the sum of two independent random currents with $A$ and $\emptyset$ boundary conditions respectively:
\[
\IP^{A}_{\drcur}(\n) =\IP^A_{\rcur} \otimes \IP^{\emptyset}_{\rcur}(\{(\n_1,\n_2) \in \currs_A \times \currs_{\emptyset} \mid \n_1 + \n_2 = \n \}), \quad \n \in \currs_A.
\]

For $u,v \in V$ and a current $\n$, we will write $u \con{\n} v$ if $u$ and $v$ are connected in $G$ by a path of edges with non-zero values of $\n$, 
and we will write~$u \ncon{\n} v$ otherwise.
Let $A$ and $B$ be contiguous sets of boundary vertices as in Corollary~\ref{cor:totpos}.
We define the event of having parallel and disjoint connections between $A$ and $B$ (see Figure~\ref{fig:parallel}) by
\[
\para_{A, B} = \{\n \in \currs_{A\cup B} \mid a_i \con{\n} b_i \text{ for all } i , \  a_i \ncon{\n} b_j \text{ for all } i\neq j\}.
\]
\begin{figure}
		\begin{center}
			\includegraphics{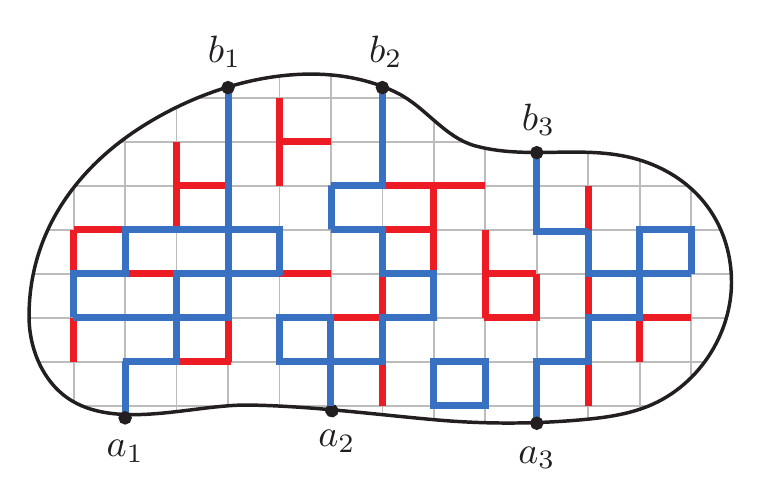}  
		\end{center}
		\caption{The coloured edges represent non-zero values of a current $\n \in \para_{A, B}$ with $A=\{a_1,a_2,a_3\}$ and 
		$B=\{b_1,b_2,b_3 \}$. The blue edges represent odd values and the red edges represent even values }
	\label{fig:parallel}
\end{figure}
As a consequence of the proof of Theorem~\ref{thm:positivity} and Corollary~\ref{cor:totpos}, we can compute the probability of $\para_{A ,B}$ under the double random current measure:
\begin{theorem} \label{thm:pairing} Let $A$ and $B$ be contiguous sets of boundary vertices as in Corollary~\ref{cor:totpos}, and let $v_1, v_2, \dots, v_{2k}$ be a counterclockwise order on $A\cup B$. 
Consider the $2k \times 2k$ skew-symmetric matrix satisfying
\[ 
{K}^{A \cup B}_{i,j} = \langle {\sigma_{v_i} \sigma_{v_j}} \rangle \qquad \text{ for } i<j.
\]
Then,
\[
\IP_{\drcur}^{A\cup B} (\para_{A, B} ) =\frac{\det M^{A,B}  }{\pf {K}^{A \cup B}},
\]
where $\pf$ denotes the Pfaffian of a skew-symmetric matrix.
\end{theorem}

\begin{example} \label{ex:1} We consider the simplest non-trivial case of Theorem~\ref{thm:pairing} when $k=2$. To this end, let $ S= \{ a,b,c,d \} \subseteq W$ be ordered counterclockwise, and let 
$\mathcal{X} = \{ \n \in \currs_{S } \mid a\con{\n} b \con{\n} c \con{\n} d \}$.
Note that by planarity of $G$, and the fact that each connected component of $\n$ contains an even number of vertices in $\partial \n$, 
the complement of $\mathcal{X}$ in $\currs_{ S}$ is $\para_{\{a,b\}, \{c,d\}} \cup \para_{\{a,d\}, \{b,c\}} $.
We have $\pf {K}^{S}  = \langle \sigma_a \sigma_b \rangle\langle \sigma_c \sigma_d \rangle +
\langle \sigma_a \sigma_d \rangle\langle \sigma_b \sigma_c \rangle -
\langle \sigma_a \sigma_c \rangle\langle \sigma_b \sigma_d \rangle$,
and hence by Theorem~\ref{thm:pairing},
\begin{align*}
&\IP_{\drcur}^{ S} (\para_{\{a,b\}, \{c,d\}} )=\frac{ \langle \sigma_a \sigma_d \rangle\langle \sigma_b \sigma_c \rangle -
\langle \sigma_a \sigma_c \rangle\langle \sigma_b \sigma_d \rangle } {\pf {K}^{ S} }, \\ 
&\IP_{\drcur}^{S} (\para_{\{a,d\}, \{b,c\}} )=\frac{ \langle \sigma_a \sigma_b \rangle\langle \sigma_c \sigma_d \rangle -
\langle \sigma_a \sigma_c\rangle\langle \sigma_b \sigma_d \rangle } {\pf {K}^{ S} }, \\
&\IP_{\drcur}^{S} (\mathcal{X}) = \frac{\langle \sigma_a \sigma_c \rangle\langle \sigma_b \sigma_d \rangle } {\pf {K}^{ S} }.
\end{align*}
We note that these formulas follow from the switching lemma of Griffiths, Hurst and Sherman~\cite{GHS} 
(see e.g.\ \cite{ADCS} for a modern treatment of the lemma). The statement of Theorem~\ref{thm:pairing} for $k>2$ does not however 
seem to be a direct consequence of this lemma.
\end{example}

\begin{remark}
For future reference, recall that the Pfaffian of the skew-symmetric matrix ${K}^{A \cup B}$ is the square root of its determinant, and it is a well known fact that it can be 
written as
\begin{align} \label{eq:nicesign2}
\pf {K}^{A \cup B} = \sum_{\pi: \textnormal{ pairing of } A \cup B} (-1)^{\textnormal{xing} (\pi)}  \prod_{\{u,v\} \in \pi}\langle \sigma_{u} \sigma_{v} \rangle,
\end{align}
where the sum is over all pairings of $A\cup B$, and where $\textnormal{xing}$ is defined as in~\eqref{eq:nicesign1}.
\end{remark}

Our last main result concerns the scaling limit at criticality of the parallel connection probability in double random currents.
The proof relies on the computation of Hongler~\cite{hongler} of the scaling limit of the boundary two-point functions themselves. Analogous and more general results 
were obtained by Kenyon and Wilson~\cite{KenWil1} for the double dimer model, multichordal loop erased walk, and grove Peano curves.

\begin{theorem}\label{thm:scalinglimit} Let $D$ be a bounded finitely-connected domain in the complex plane with a piecewise $C^1$ boundary. 
Let $\partial^{\textnormal{s}} D$ be the straight part of the boundary, i.e., the part composed of intervals 
parallel either to the real or imaginary axis. 
Let $A=\{ a_1,a_2,\ldots,a_k \} \subset \partial^{\textnormal{s}} D$ and $B=\{ b_1, b_2, \ldots b_k\} \subset \partial^{\textnormal{s}} D$ be such that
$a_1,\ldots,a_k, b_k,\ldots,b_1$
is a counterclockwise ordering of $A\cup B$ around the outer boundary of $D$.
For $\epsilon >0$, let $D_{\epsilon}$ be the maximal connected component of the rescaled square nearest-neighbor lattice
$\epsilon \mathds{Z}^2$ contained in $D$. Consider the double random current measure $\IP^{A_{\epsilon} \cup B_{\epsilon}}_{\drcur}$ on 
$D_{\epsilon}$ with homogeneous critical coupling constants $J_e = \frac12 \log (\sqrt2 +1) $, where $A_{\epsilon}$ and $B_{\epsilon}$
are sets of vertices on the outer face of $D_{\epsilon}$ approximating $A$ and $B$. 
Then, the following limit exists
\[
p_{A,B}(D) : = \lim_{\epsilon \to 0} \IP_{\drcur}^{A_{\epsilon}\cup B_{\epsilon}} (\para_{A_{\epsilon}, B_{\epsilon}} ).
\]
Moreover, if $\mathbb{H}$ is the upper half-plane and $\varphi : D \to \mathbb{H}$ is a conformal equivalence between $D$ and $\mathbb{H}$, then
\[
p_{A,B}(D) = \frac{\det \tilde {M}^{A,B}  }{\pf \tilde {K}^{A\cup B}},
\]
where $\tilde M^{A,B }$ is a $k \times k$ matrix, and $ \tilde K^{A\cup B} $ is a $2k \times 2k$ skew-symmetric matrix satisfying
\[
\tilde M^{A,B }_{i,j} =  \frac{1}{|\varphi(a_i)-\varphi(b_j)|}, \quad \text{and} \quad  \tilde K^{A\cup B}_{i,j} = \frac{1}{|\varphi(v_i)-\varphi(v_j)|} \quad \text{for } i < j,
\]
where $v_1 ,v_2 , \dots , v_{2k} $ is a counterclockwise order on $A\cup B$.
\end{theorem}

\section{Random currents} \label{sec:currents}
In this section $G=(V,E)$ is a finite (not necessarily planar) graph. Note that the definitions of the Ising model and the random current model generalize verbatim to arbitrary finite graphs.

Let $A\subseteq V$ be a set of even cardinality (possibly empty).
With a current $\n \in \currs_A$, we associate a pair of sets of edges $\setm(\n)=(\setm_1(\n),\setm_2(\n))$, where 
\[
\setm_1(\n)=\{e\in E \mid \n_e \text{ is odd}\}, \quad \text{and} \quad \setm_2(\n)=\{e\in E \mid \n_e \text{ is even},\ \n_e \neq 0\}.
\] 
Let $\mathcal{E}_A$ be the collection of sets of edges for which $A$ is the set of vertices with odd degree in the induced subgraph.
One can see that
\[
\sets_A := \{\set(\n) \mid \n\in \currs_A\}= \{(\set_1,\set_2) \mid \set_1 \in \mathcal{E}_A,\   \set_2 \subseteq E \setminus \set_1\}.
\]
We will often identify $\set \in \sets_A$ with the set $\set_1\cup \set_2 \subseteq E$.
Note that for $u,v\in V$, $u \con{\n} v$ if and only if $u$ and $v$ belong to the same connected component of $\setm(\n)$.

Let $x_e =\tanh J_e$ and $y_e=(\cosh J_e)^{-1}$. The next result describes the measure $\overline{\IP}^A_{ \rcur}$ on $\sets_A$ induced from the random current measure.

\begin{lemma} \label{lem:singleinduced}
The probability of $\set \in \sets_A$ induced from the random current measure is given by 
\[
\overline{\IP}^A_{ \rcur}(\set) = \frac 1{\bar{Z}^A_{ \rcur}} \prod_{e\in \set_1} x_e \prod_{e \in \set_2}( 1-y_e)  \prod_{e\in E\setminus \set} y_e,
\]
where $\bar{Z}^A_{ \rcur}= Z^A_{\rcur}   \prod_{e\in E} y_e $.
\end{lemma}

\begin{proof}  We have
\begin{align*}
\overline{\IP}^A_{ \rcur}(\set) &= \sum_{\n \in \currs_A: \ \setm(\n)=\set}  \IP^A_{ \rcur} (\n)  \\
 &=\frac 1{{Z}^A_{ \rcur}} \sum_{\n \in \currs_A: \ \setm(\n)=\set} \prod_{e \in E} \frac{ (J_e)^{\n_e}}{\n_e!}  \\
 &= \frac 1{{Z}^A_{ \rcur}}\prod_{e\in \set_1}\sum_{k=1}^{\infty} \frac{(J_e)^{2k-1}}{(2k-1)!} \prod_{e \in \set_2}\sum_{k=1}^{\infty} \frac{(J_e)^{2k}}{(2k)!}  \\
 &=  \frac 1{{Z}^A_{ \rcur}} \prod_{e\in \set_1} \sinh J_e \prod_{e \in \set_2}( \cosh J_e-1) \\
& =   \frac1{\bar{Z}^A_{ \rcur}} \prod_{e\in \set_1} x_e \prod_{e \in \set_2}( 1-y_e)  \prod_{e\in E\setminus \set} y_e  . \qedhere
\end{align*}
\end{proof}

For $\set\subseteq E$, let $\mathcal{E}_{\emptyset}(\set) = \mathcal{E}_{\emptyset} \cap \{\set' \mid \set'\subseteq \set \}$.
The main result of this section gives a formula for the probability measure $\overline{\IP}^{A}_{ \drcur}$ on $\sets_A$ induced from the double random current measure.

\begin{theorem} \label{thm:dcurrformula} The probability of $\set \in \sets_A$ induced from the double random current measure is given by
\[
\overline{\IP}^{A}_{ \drcur}(\set)= \frac1{\bar{Z}^{A}_{ \drcur}}  |\mathcal{E}_{\emptyset}(\set)| \prod_{e\in \set_1} x_e  \prod_{e\in \set_2} x^2_e 
   \prod_{e\in E\setminus \set}   (1-x^2_e),
\]
where $\bar{Z}^{A}_{ \drcur}=Z^A_{\rcur}  Z^{\emptyset}_{\rcur}  \prod_{e\in E} (1-x^2_e)$.
\end{theorem}

\begin{proof}
Let $\n_1 \in \currs_A$, $\n_2 \in \currs_{\emptyset}$, $\n=\n_1 + \n_2$, and $\set^1= \setm(\n_1)$, $\set^2=\setm(\n_2)$, $\set = \setm(\n)$. 
Note that $e\in \set_1$ if and only if $e\in \set^1_1 \triangle \set^2_1$. 
Therefore, by Lemma~\ref{lem:singleinduced}, each $e \in \set_1$ carries a multiplicative weight of $x_e[(1-y_e)+y_e]=x_e$
in the double random current measure. Similarly, $e\in \set_2$ if and only if $e \in \set^1_1 \cap \set^2_1$ or $e \in (\set^1_2 \cup \set^2_2)\setminus (\set^1_1 \cup \set^2_1) $.
By Lemma~\ref{lem:singleinduced}, the multiplicative contribution of $e\in \set_2$ is hence $x_e^2= (1-y_e+y_e)^2-y_e^2$ in both cases.
The contribution of an edge $e \in E \setminus \set$ is then $1-x_e^2$. It is now enough to show that there are exactly 
$|\mathcal{E}_{\emptyset}(\set)|$ choices of $\set^1_1$ and $\set^2_1$ such that $\set^1_1 \triangle \set^2_1=\set_1$ and
$\set^1_1 \cap \set^2_1\subset \set_2$. Indeed, this is equivalent to freely choosing $\set^2_1 \in \mathcal{E}_{\emptyset}(\set)$ and setting
$\set^1_1 = (\set_1 \setminus \set^2_1) \cup (\set^2_1 \cap \set_2) \in \mathcal{E}_{A}$.
\end{proof}

We note that one can express $| \mathcal{E}_{\emptyset}(\set) | $ in terms of the number of vertices, edges, and connected components of $\set$ (see Lemma~\ref{lem:eqcard}).

\section{Random alternating flows} \label{sec:flows}
In this section we assume that $G$ is planar, and we explain how the double random current measure is related to a measure on alternating flows. 
To this end, we define a directed modification $\vec G=(V\cup W^{\source} \cup W^{\sink},\vec E)$ of $G$ as follows.
Each edge $e\in E$ is replaced by three directed edges $\vec e \in \vec E$: 
one middle edge $\vec e_m$, and two side edges $\vec e_{s1}, \vec e_{s2}$ (see Figure~\ref{fig:directedgraph}). The side edges lie on opposite sides
of $\vec e_m$ and have the opposite orientation to $\vec e_m$. The orientation of each middle edge is chosen arbitrarily, 
and the edge weights are given by
\begin{align} \label{eq:edgeweights}
x_{\vec e_m} = \frac{\sinh 2J_e}2 = \frac{x_e}{1-x_e^2}, \qquad x_{\vec e_{s1}}= x_{\vec e_{s2}} = \frac{\tanh J_e}2 = \frac{x_e}2.
\end{align}
Moreover, for each $w\in W$, two vertices $w^{\source} \in W^{\source}$ and $w^{\sink} \in W^{\sink}$, and two directed edges $(w,w^{\sink}), (w^{\source},w) \in \vec E$ 
are placed in the external face of $G$ so that
$w^{\sink}$ becomes a sink and $w^{\source}$ becomes a source. The weights of these edges  are set to $1$.
The vertices are placed in such a way that $w^{\sink}$ is comes immediately before (resp.\ after) $w^{\source}$ if $w\in W_{\circ}$ (resp.\ if $w\in W_{\bullet}$) in the counterclockwise order around the 
external face of $G$ (see Figure~\ref{fig:directedgraph}).

 \begin{figure}
		\begin{center}
			\includegraphics{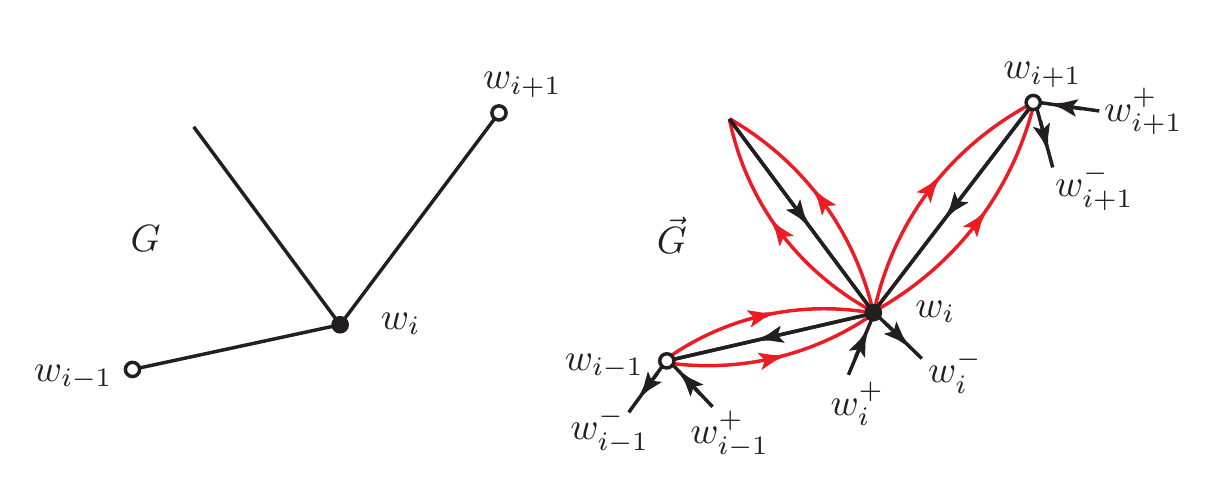}  
		\end{center}
		\caption{A choice of orientations of edges in a directed modification $\vec G$ of $G$. For each boundary vertex $w$, a source $w^{\source}$ and a sink $w^{\sink}$ are added in an order depending on whether $w\in W_{\circ}$ 
		or $w\in W_{\bullet}$  }
	\label{fig:directedgraph}
\end{figure}

A flow on $\vec G$ is a set of edges $F \subseteq \vec E$ such that each vertex has the same number of ingoing and outgoing edges in $F$.
An alternating flow is a flow such that for each $v\in V$, the edges of the flow alternate in orientation around $v$ (edge oriented towards $v$ or away from~$v$).
For $U\subseteq W$, we will write $U^{\source} \subseteq W^{\source}$ and $U^{\sink} \subseteq W^{\sink}$ for the the corresponding sets of sources and sinks in $\vec G$.
For two (possibly intersecting) sets $A ,B\subset W$ of equal cardinality, let 
\[
\flows_{A,B} = \{  \text{alternating flow } F \text{ on } \vec G \mid  V(F)\cap W^{\source} =A^{\source}, V(F)\cap W^{\sink} =B^{\sink} \},
\]
where $V(F)$ is the set of vertices incident on at least one edge of $F$.

Following Talaska~\cite{talaska}, 
we define the weight of an alternating flow $F \in \flows_{A,B} $ by 
\begin{align} \label{def:flow}
\weight(F) = 2^{|A|+|F|-|V(F)|}\prod_{\vec e \in F} x_{\vec e},
\end{align}
and we consider the probability measure on alternating flows with boundary conditions $A,B$ given by
\begin{align*}
\IP^{A,B}_{\aflow} (F) = \frac{\weight(F)}{Z^{A,B}_{\aflow}}, \qquad F  \in \flows_{A,B},
\end{align*}
where 
$Z^{A,B}_{\aflow} = \sum_{F\in \flows_{A,B}} \weight(F)$ is the partition function. 
Note that in~\cite{talaska} general oriented planar graphs are considered, and the connection with double random currents described in this section is realized by our
particular choice of $\vec G$ and its edge weights.

There are four different types of local (interior) edge configurations in an alternating flow $F$ on $\vec G$ according to the
total amount of flow at the endpoints (the number of incoming edges minus the number of outgoing edges) and the orientation of the edges (see Figure~\ref{fig:types}).
Note that it is not possible that $F$ contains $\vec e_{R1}$ and $\vec e_{R2}$ but not $\vec e_{B}$ since then the alternating condition is violated. Also note
that there are three different configurations of type (1a) which are indistinguishable from the point of view of alternating flows -- they can be interchanged without forcing any other changes in the flow. 
 \begin{figure}
		\begin{center}
			\includegraphics{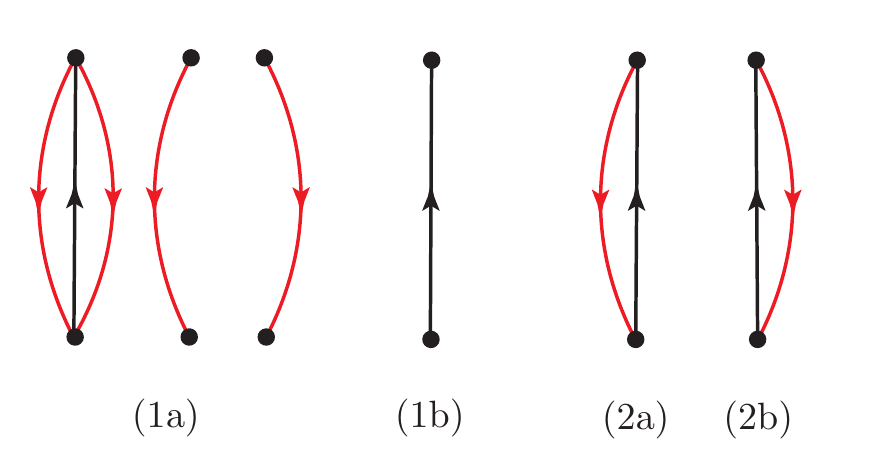}  
		\end{center}
		\caption{Four types of local edge configurations in alternating flows on $\vec G$
		 }
	\label{fig:types}
\end{figure} 

With each $F \in  \flows_{A,B}$, as in Section~\ref{sec:currents}, we associate a pair of sets $\setf(F)=(\setf_1(F), \setf_2(F))$,
where $\setf_1(F)$ is the set of edges $e \in E$ such that the local configuration of $F$ at $e$ is of type (1a) or (1b), and where
$\setf_2(F)$ is the set of edges $e \in E$ with the local configuration of type (2a) or (2b).
We denote by $\sets_{A,B}$ the image of $\flows_{A,B}$ under this map. 
Note that $\setf_1(F) \in \mathcal{E}_{A \triangle B}$ since, by the flow condition, one can split $\set_1(F)$ into a sourceless part in $\mathcal{E}_{\emptyset}$ and a collection of $|A|$ 
edge-disjoint directed paths starting at $A$ and ending at~$B$. In particular, every $v \in A \cap B$ is the starting and ending point of exactly one such path, and hence, 
the degree of $\set_1(F)$ at $v$ is even.
This means that $\sets_{A,B} \subseteq \sets_{A\triangle B}$, and hence the map above induces a probability measure
$\overline{\IP}^{A,B}_{{ \textnormal{a-flow}} }$ on $\sets_{A\triangle B}$, which is supported on $\sets_{A,B}$. 

The main result of this section casts this measure into a form 
related to that of the induced double random current measure from Theorem~\ref{thm:dcurrformula}.

\begin{theorem}  \label{thm:flowformula}
Let $A , B \subset W$ be such that $|A|=|B|$, and let $\vec G$ be one of the directed modifications of $G$ described above.
Let $\set \in\sets_{A, B}$, and let $k'(\set)$ be the number of connected components of $\set$ that contain a vertex in $A \cup B$.
Then, the probability of $\set $ induced from the random alternating flow measure is given by
\[
\overline{\IP}^{A,B}_{{ \textnormal{a-flow}} }(\set) = \frac{1  }{\bar{Z}^{A,B}_{ \aflow}} 
2^{|A|-k'(\set)} |\mathcal{E}_{\emptyset}(\set)| \prod_{e\in \set_1} x_e  \prod_{e\in \set_2} x^2_e 
   \prod_{e\in E\setminus \set}  (1-x^2_e),\]         
where $\bar{Z}^{A,B}_{ \aflow}= Z^{A,B}_{\aflow}  \prod_{e\in E} (1-x_e^2) $. In particular, $\overline{\IP}^{A,B}_{{ \textnormal{a-flow}} } =\overline{\IP}^{B,A}_{{ \textnormal{a-flow}} }$.
\end{theorem}

Before proving the theorem, we need to recall a standard result about the cardinality of $\mathcal{E}_{\emptyset}(\set)$. We give a proof for completeness.
For $\set \subseteq E$, we denote by $k(\set)$ the number of connected components of $\set$.

\begin{lemma} \label{lem:eqcard}
Let $\set\subseteq E$ be such that $\mathcal{E}_{\emptyset}(\set)$ is nonempty. Then,
\[
|\mathcal{E}_{\emptyset}(\set)|=2^{|\set|-|V(\set)|+k(\set)}.
\] 
\end{lemma}
\begin{proof} 
Consider a maximal spanning forest $T\subseteq \set$ of $\set$, and note that $|\set \setminus T| =|\set |-|V(\set)|+k(\set)$.
It is hence enough to construct a bijection between $\mathcal{E}_{\emptyset}(\set)$ and the set of subsets of $\set \setminus T$.
To this end, to each $e \in \set \setminus T$ we assign the unique cycle $C_e \subseteq (\set\setminus T) \cup \{e\}$. Note that $C_e \in \mathcal{E}_{\emptyset}(\set)$. 
The bijection is given by assigning to each $\set' \in \mathcal{E}_{\emptyset}(\set)$ the set  $\set'\cap(\set\setminus T)$, and its inverse by assigning to each 
$ \{e_1,\ldots, e_l\}\subseteq \set \setminus T$ the set $C_{e_1}\triangle \cdots \triangle C_{e_l}\in \mathcal{E}_{\emptyset}(\set)$.
\end{proof}

\begin{proof}[Proof of Theorem~\ref{thm:flowformula}]
Let $\bar \set$ be a modification of $\set$ seen as a subset of $E$ where each edge in $\set_2$ is replaced by two parallel edges. By an alternating orientation of
$\bar \set$ we mean an assignment of orientations to the edges of $\bar \set$ such that for each $v\in V$, the edges alternate in orientation around $v$.
Let $F$ be such that $\setf(F)=\set$, and let $\ec $ be its equivalence class under identifying the three local edge configurations of type (1a).  
Note that such 
equivalence classes are in one-to-one 
correspondence with alternating orientations of $\bar \set$.
We claim that the number of such alternating orientations is
\begin{align} \label{eq:compnumb}
2^{k(\set)-k'(\set)}.
\end{align}
Indeed, consider a connected component $\ccomp$ of $\bar \set$ that does not contain a vertex in $A \cup B$.
We claim that $ \kappa$ has exactly $2$ different alternating orientations.
To see this, consider the outer boundary $\partial_{\text{out}}  \ccomp$ of $ \ccomp$, i.e., the set of edges in $ \ccomp$ incident on the 
unbounded face of $ \ccomp$. Because of the alternating condition, the choice of the orientation for one edge determines
the orientation of all other edges in $\partial_{\text{out}} \ccomp$. Moreover, there are no conflicts of orientations since each vertex of $ \ccomp$ has even degree, 
and hence $\partial_{\text{out}}  \ccomp$ can be deformed to a cycle by splitting vertices incident on more than $2$ edges from $\partial_{\text{out}}  \ccomp$. Then, the 
orientations of all edges agree with the clockwise or anticlockwise order on the cycle. After orienting 
$\partial_{\text{out}} \ccomp$, one has to consider all connected components of $ \ccomp \setminus \partial_{\text{out}}  \ccomp$,
and repeat the reasoning with the only difference being that there is no more freedom of choice of the orientation (if $\partial_{\text{out}}  \ccomp$ was oriented clockwise (resp. counterclockwise), 
then all outer boundaries of the connected components of $ \ccomp \setminus \partial_{\text{out}} \ccomp$ have to be oriented counterclockwise (resp. clockwise)). One repeats this 
procedure until all edges of $ \ccomp$ are oriented. On the other hand, there is no freedom of orientation for components containing a vertex in $A \cup B$ 
since each such component has to be oriented from a vertex in $A$ to a vertex in $B$. Therefore \eqref{eq:compnumb} holds true.

We now turn to the total weight $\weight( \ec)$ of each equivalence class $\ec$, which we define to be the sum 
of weights of all flows in $\ec$. By \eqref{eq:edgeweights}, the total multiplicative contribution to $\weight( \ec)$ of configurations of type (1a) is
\[
2^3\Big(\frac {x_e}2\Big)^2 \frac{x_e}{1-x_e^2} +2x_e= \frac{2x_e}{1-x_e^2} ,
\]
where we took the factor $2^{\text{\# edges}}$ and did not take the factor $2^{|A|-\text{\# vertices}}$ from \eqref{def:flow} into account.
This contribution is therefore equal to the contribution of a configuration of type (1b). The contributions of configurations of type (2a) and (2b) also agree and are equal to
\[
2^2\Big(\frac {x_e}2\Big)\frac{x_e}{1-x_e^2} = \frac{2x^2_e}{1-x_e^2}.
\]
Hence, by \eqref{def:flow}, we can write
\begin{align*}
\weight( \ec) & =  2^{|A|-|V(\set)|}\prod_{ e \in \set_1} \frac{2x_e}{1-x_e^2}\prod_{ e \in \set_2} \frac{2x^2_e}{1-x_e^2} \\
                    &  =  2^{|A|+|\set|-|V(\set)|}\hspace*{-0.2cm}\prod_{ e \in \set_1} \frac{x_e}{1-x_e^2}\prod_{ e \in \set_2} \frac{x^2_e}{1-x_e^2}.
\end{align*}

Note that the map $\setf$ is constant on each equivalence class $\ec$, and hence, with a slight abuse of notation, we can write $\setf(\ec)$ for the value of $\setf$ 
evaluated at any representative of $\ec$.
Combining all the previous observations, and using the definition of $\IP^{A,B}_{{ \textnormal{a-flow}} }$, we can write
\begin{align*}
&\overline{\IP}^{A,B}_{{ \textnormal{a-flow}} }(\set) = \sum_{F:\  \setf(F)=\set }\IP^{A,B}_{{ \textnormal{a-flow}} }(F) \\
&= \frac1{Z^{A,B}_{\aflow} } \sum_{F:\  \setf(F)=\set } \weight(F) \\ 
&=  \frac1{Z^{A,B}_{\aflow} } \sum_{\ec:\ \setf(\ec)=\set}\weight(\ec) \\
 &=  \frac1{Z^{A,B}_{\aflow} } \sum_{\ec:\ \setf(\ec)=\set}  2^{|A|+|\set|-|V(\set)|}  \prod_{ e \in \set_1} \frac{x_e}{1-x_e^2}\prod_{ e \in \set_2} \frac{x^2_e}{1-x_e^2}\\
&= \frac1{\bar{Z}^{A,B}_{ \aflow}} \sum_{\ec:\ \setf(\ec)=\set  } 2^{|A|+|\set|-|V(\set)|}  
\prod_{ e \in \set_1}x_e \prod_{ e \in \set_2} x^2_e\prod_{e\in E \setminus \set} (1-x_e^2) \\
&= \frac1{\bar{Z}^{A,B}_{ \aflow}} 2^{k(\set)-k'(\set)} 2^{|A|+|\set|-|V(\set)|} \prod_{ e \in \set_1}x_e \prod_{ e \in \set_2} x^2_e\prod_{e\in E \setminus \set} (1-x_e^2) \\
&= \frac{1 }{\bar{Z}^{A,B}_{ \aflow}}  2^{|A|-k'(\set)} |\mathcal{E}_{\emptyset}(\set)| \prod_{e\in \set_1} x_e  \prod_{e\in \set_2} x^2_e 
   \prod_{e\in E\setminus \set}  (1-x^2_e),
\end{align*}
where the second to last equality follows from \eqref{eq:compnumb}, and the last one from Lemma~\ref{lem:eqcard}.
\end{proof}

In two special cases the measures induced from double random currents and alternating flows are the same.
\begin{corollary} \label{cor:empty}
We have that $\overline{\IP}^{\emptyset}_{{ \drcur} }=\overline{\IP}^{\emptyset,\emptyset}_{{ \textnormal{a-flow}} }$ and $\overline{\IP}^{\{a,b\}}_{{ \drcur} }=\overline{\IP}^{\{a\},\{b\}}_{{ \textnormal{a-flow}} }$
for any $a,b \in W$, $a\neq b$.
\end{corollary}
\begin{proof}
By Lemma~\ref{lem:revflowprop}, we have that $\sets_{\emptyset,\emptyset}= \sets_{\emptyset}$ and $\sets_{\{a\},\{b\}}= \sets_{\{a,b\}}$, and the statement follows from 
Theorems~\ref{thm:dcurrformula} and~\ref{thm:flowformula}.
\end{proof}

\section{Proofs of main results} \label{sec:proofs}

We need to state a few necessary lemmas.
The first one goes back to the work of Griffiths, Hurst and Sherman~\cite{GHS}, and we give a proof for completeness.
We will write $ Z^{a,b} = Z^{\{ a,b\}}$, and we define $ Z_{\rcur}^{a,b} = Z^{\emptyset}_{\rcur}$ for $a=b$.
\begin{lemma} \label{lem:currtwopoint}
For $a,b \in V$, we have
\[
\langle \sigma_a \sigma_b \rangle = \frac{Z^{a,b}_{\rcur}}{Z^{\emptyset}_{\rcur}}.
\]
\end{lemma}
\begin{proof} If $a=b$, then the equality is trivial.  Otherwise, for a vertex $v$ and a current $\n$, we define $\n_v = \sum_{u : \ \{u,v\} \in E} \n_{\{u,v\}}$, and we have
\begin{align*} 
\langle \sigma_a \sigma_b \rangle &=\frac{\sum \limits_{\sigma \in  \{ -1,1\}^V}  \sigma_a \sigma_b 
\prod \limits_{\{u,v\}\in E} \exp(J_{\{u,v \}} \sigma_u \sigma_v)}{\sum \limits_{\sigma \in  \{ -1,1\}^V}  
\prod \limits_{\{u,v\}\in E} \exp(J_{\{u,v \}} \sigma_u \sigma_v)} \\
&= \frac{\sum \limits_{\sigma \in  \{ -1,1\}^V}  \sigma_a \sigma_b  
\prod \limits_{\{u,v\}\in E} \sum \limits_{k=0}^{\infty} (J_{\{u,v \}})^k \sigma^k_u \sigma^k_v/k!}{\sum \limits_{\sigma \in  \{ -1,1\}^V}  
\prod \limits_{\{u,v\}\in E}  \sum \limits_{k=0}^{\infty} (J_{\{u,v \}})^k \sigma^k_u \sigma^k_v/k!} \\
&= 
 \frac{\sum \limits_{\sigma \in  \{ -1,1\}^V}  \sigma_a \sigma_b  
 \sum \limits_{\n \in \currs} w(\n) \prod \limits_{v\in V} \sigma^{\n_v}_v }{\sum \limits_{\sigma \in  \{ -1,1\}^V} 
 \sum \limits_{\n \in \currs} w(\n) \prod \limits_{v\in V} \sigma^{\n_v}_v } \\
 &= 
 \frac{\sum \limits_{\sigma \in  \{ -1,1\}^V}  
 \sum \limits_{\n \in \currs} w(\n) \prod \limits_{v\in V} \sigma^{\n_v+\mathds{1}\{v\in \{a,b \}\}}_v }{\sum \limits_{\sigma \in  \{ -1,1\}^V} 
 \sum \limits_{\n \in \currs} w(\n) \prod \limits_{v\in V} \sigma^{\n_v}_v } \\
 &= 
  \frac{
 2^{|V|} \sum \limits_{\n \in \currs_{\{ a,b \}}} w(\n) }{
 2^{|V|} \sum \limits_{\n \in \currs_{\emptyset}} w(\n)  } \\
 & = \frac{Z^{a,b}_{\rcur}}{Z^{\emptyset}_{\rcur}}.
\end{align*}
The second to last equality holds true since the only currents that survive the summation over the symmetric set $\{-1,1\}^{|V|}$ 
are the ones for which the exponent of $\sigma_v$ is even at every vertex. 
\end{proof}
The next lemma expresses the two-point spin correlation function as a ratio of partition functions of alternating flows.
Recall that we defined $\sets_{A,B} \subseteq \sets_{A \triangle B}$ to be the image of $\flows_{A,B}$ under the map $\setm$ from Section~\ref{sec:flows}.
\begin{lemma} \label{lem:flowtwopoint}
For $a,b \in W$, we have
\[
\langle \sigma_a \sigma_b \rangle =   \frac{Z^{a,b}_{\aflow}}{Z^{\emptyset}_{\aflow}}.
\]
\end{lemma}
\begin{proof}
Note that for each $F \in \flows_{\{a\},\{b\}}  $, the graph $\setm(F)$ contains exactly $k'=1$ connected component that connects $a$ and $b$. Moreover, by Lemma~\ref{lem:revflowprop} we have that
$\sets_{\{a\},\{b\}} = \sets_{\{a\} \triangle \{ b\}}$, and $\sets_{\emptyset,\emptyset} = \sets_{\emptyset}$.
Hence, by comparing the formulas in Theorem~\ref{thm:dcurrformula} and~\ref{thm:flowformula}, and using
Lemma~\ref{lem:currtwopoint}, we get
\[
\frac{Z^{a,b}_{\aflow}}{Z^{\emptyset}_{\aflow}} = 2^{|\{ a\}|}2^{-k'}\frac{Z^{a,b}_{\rcur}Z^{\emptyset}_{\rcur}}{Z^{\emptyset}_{\rcur}Z^{\emptyset}_{\rcur}} = \frac{Z^{a,b}_{\rcur}}{Z^{\emptyset}_{\rcur}}=
\langle \sigma_a \sigma_b \rangle. \qedhere
\]
\end{proof}

The next lemma describes an alternating property of sources and sinks in a connected component of an alternating flow.
\begin{lemma} \label{lem:flowprop}
Let $F \in  \flows_{A,B}$, and let $\ccomp\subset F$ be one of its connected components. 
Then, the sources and sinks of $\ccomp$ interlace, i.e., as one goes around the external face, the vertices in $V(\ccomp) \cap (A^{\source} \cup B^{\sink})$ alternate
between the sources in $A^{\source}$ and the sinks in $B^{\sink}$. 
\end{lemma}
\begin{proof}
Note that $\ccomp$ is an alternating flow itself.
Take a source vertex $v \in V(\ccomp) \cap A^{\source} $, and traverse the edges of the external face of $\ccomp$ in a counterclockwise order.
Since at each vertex in $V$ that you visit, you take the rightmost possible turn, and since $\ccomp$ is alternating, all the directed edges are aligned with the direction of traversal until
you encounter the consecutive vertex in $v'\in V(\ccomp) \cap (A^{\source} \cup B^{\sink})$.  Since $v'$ is either a sink or a source, and since there exists an edge
directed towards $v'$, it must be a sink. An analogous argument can be made when starting at $v'$ but this time the edges are directed against the direction of traversal. 
Hence, the next vertex encountered in $V(\ccomp) \cap (A^{\source} \cup B^{\sink})$ is again a source, and the lemma is proved.
\end{proof}

The following is the reverse of the lemma above.
\begin{lemma} \label{lem:revflowprop}
Let $\set \in \sets_{A \triangle B}$ be such that for each connected component $\ccomp$ of $\set$, the vertices in $A^{\source} \cup B^{\sink}$ 
that are adjacent to $V(\ccomp)$ alternate between the sources in $A^{\source}$ and the sinks in $B^{\sink}$ as one goes around the external face. 
Then, there exists $F \in \flows_{A,B}$ such that $\setm(F)=\set$.
\end{lemma}
\begin{proof}
Fix a connected component $\ccomp$ of $\set$. Similarly to the proof of Theorem~\ref{thm:flowformula}, we consider a graph $\bar \ccomp $ 
where each edge of $\ccomp \cap \set_2$ is replaced by two parallel edges, and moreover,
the edges connecting each vertex in $V(\ccomp) \cap W$ to the corresponding vertices in $A^{\source} \cup B^{\sink}$ are added. It is now enough to show that there exists an alternating orientation of 
$\bar \ccomp$ where each vertex in $V(\bar \ccomp) \cap A^{\source}$ becomes a source and each vertex in $V(\bar \ccomp) \cap B^{\sink}$ becomes a sink. By the assumption on $\ccomp$, we can 
add (in one of two possible ways) $|V(\bar \ccomp) \cap (A^{\source}\cup B^{\sink})|/2$ vertex-disjoint edges connecting in 
pairs consecutive vertices in $V(\bar \ccomp) \cap A^{\source}$ and $V(\bar \ccomp) \cap B^{\sink}$. We call the resulting graph $\bar \ccomp'$.
It is now enough to construct an alternating orientation of $\bar \ccomp'$ such that each of the additional edges is directed from a vertex in $V(\bar \ccomp) \cap B^{\sink}$ 
to a vertex in $V(\bar \ccomp) \cap A^{\source}$, 
and then restrict the orientation to $\bar \ccomp$. To this end, note that each vertex
of $\bar \ccomp'$ has even degree. Therefore, as in the proof Theorem~\ref{thm:flowformula}, the outer boundary of $\bar \ccomp'$ can be deformed to a cycle, and can be directed 
clockwise or counterclockwise so that the additional edges are directed properly. We can then orient the rest of $\bar \ccomp'$ consistently as in 
Theorem~\ref{thm:flowformula}, and the lemma is proved.
\end{proof}

The next lemma is an adaptation of a result of~\cite{talaska}.
\begin{lemma} \label{lem:tposflow}
\cite[Corollary 4.3]{talaska} \label{lem:talaska} Let ${N}^{A, B}$ be as in Theorem~\ref{thm:positivity}. We have
\[ 
\det {N}^{A, B} = \frac{Z^{A,B}_{\aflow}}{Z^{\emptyset}_{\aflow}}.
\]
\end{lemma}
\begin{proof}
The proof is a matter of translation of the results of \cite{talaska} to our setting. To this end,
consider the boundary measurement matrix $M$ from Definition 2.4 of \cite{talaska} defined for $\vec G$ with weights as in Section~\ref{sec:flows}. 
By Corollary~5.3 of~\cite{talaska}, the entry corresponding to source $a^{\source}$ and sink $b^{\sink}$ is equal to ${Z^{a,b}_{\aflow}}/{Z^{\emptyset}_{\aflow}}$,
and hence, by Lemma~\ref{lem:flowtwopoint}, to $\langle \sigma_a \sigma_b \rangle$. Therefore, the boundary measurement matrix for $\vec G$ is the matrix of Ising boundary spin correlation functions.

By Corollary~4.3 of~\cite{talaska}, it is hence enough to check that the signs in $N^A$ agree with those in the matrix in Corollary~4.3 of~\cite{talaska}
where only the columns corresponding to sinks are considered.
In~\cite{talaska}, the sign of the entry corresponding to source $a^{\source}$ and sink $b^{\sink}$ is negative if the number of sources strictly between $a^{\source}$ and $b^{\sink}$
in a fixed clockwise order is odd, and it is positive otherwise.
By our choice, for every $w\in W_{\circ}$ (resp.\ $w\in W_{\bullet}$), the sink $w^{\sink}$ immediately precedes (resp.\ succeeds) the source 
$w^{\source}$ in the counterclockwise order. One can easily check that as a result, $s(i,j)$ from the definition of ${N}^{A, B}$ 
counts the number of sources strictly between the source $w^{\source}_{l_i} \in A^{\source}$ and the sink $w^{\sink}_{j} \in B^{\sink}$.
Therefore, the signs in the definition of $N^A$ agree with those in \cite{talaska} (up to changing a clockwise to a counterclockwise order).
Hence, the result follows from Corollary~4.3 of~\cite{talaska}.
\end{proof}

\begin{remark}
The proof of Corollary~4.3 of~\cite{talaska} uses the signed walk interpretation of the boundary measurement matrix of Postnikov~\cite{postnikov}.
One can compare this with the Kac--Ward representation of the planar Ising model~\cite{KacWard}, where
the Ising boundary two-point functions can be expressed as partition functions of signed non-backtracking walks on the undirected graph $G$ 
(see e.g.\ \cite{KLM}, or \cite{Lis} for a concise proof of the Kac--Ward formula). 
\end{remark}

We are now ready to prove our first two main results.
\begin{proof}[Proof of Theorem~\ref{thm:positivity}]
It is a direct consequence of Lemma~\ref{lem:tposflow}. The interpretation of empty paths as counterclockwise or 
clockwise loops follows from the way the sources and sinks are placed around the external face 
(see Section~\ref{sec:flows}).
\end{proof}

\begin{proof}[Proof of Corollary~\ref{cor:totpos}]
Note that every square submatrix of $M^{A,B}$ satisfies the assumptions of Corollary~\ref{cor:totpos}. Hence, it is enough to prove that $\det M^{A,B} \geq 0$, 
and $\det M^{A,B} > 0$ if and only if there are $k$ vertex disjoint paths connecting $A$ and $B$. To this end, note that by~\eqref{eq:nicesign}, \eqref{eq:nicesign1} and Lemma~\ref{lem:tposflow}, 
\[
\det M^{A,B} = \det N^{A,B} = \frac{Z^{A,B}_{\aflow}}{Z^{\emptyset}_{\aflow}} \geq 0.
\]
Furthermore, by Lemma~\ref{lem:flowprop} and~\ref{lem:revflowprop},
$Z^{A,B}_{\aflow} >0$ if and only if there are $k$ vertex disjoint paths connecting $A$ and $B$. Indeed,  in a flow $F \in \flows_{A,B}$ there cannot be a connected component
that contains more than $2$ vertices in $A\cup B$ since then the alternating condition from Lemma~\ref{lem:flowprop} would be violated. On the other hand, if there exists
$k$ vertex disjoint paths connecting $A$ and $B$, then by Lemma~\ref{lem:revflowprop}, there exists a flow in $\flows_{A,B}$ which maps to these paths under~$\setf$.
\end{proof}

Before proving the rest of our main results, we will need a classical Pfaffian formula of Groeneveld, Boel and Kasteleyn~\cite{GBK}. 
We present here a different proof involving the connection between double random currents and alternating flows.
For yet another proof involving the dimer model, see the recent treatment of the combinatorics of the planar Ising model~\cite{CCK}.

\begin{lemma}\cite[Theorem A]{GBK} \label{lem:gbk} Let $K^{A \cup B}$ be as in Theorem~\ref{thm:pairing}. We have
\[
\pf {K}^{A \cup B}=  \frac{Z^{A\cup B}_{\rcur}}{Z^{\emptyset}_{\rcur}}.
\]
\end{lemma}
\begin{proof}
By \eqref{eq:nicesign1}, for any $A' \subset A\cup B$ with $|A'|=|A|$ and $B'=(A\cup B)\setminus A'$, we have
\[
\det N^{A',B'} = \sum_{\pi: \textnormal{ bijection } A' \to B'} (-1)^{\textnormal{xing}(\pi)} \prod_{\{a,b\} \in \pi}\langle \sigma_{a} \sigma_{b} \rangle
\]
Recall that a bijection between $ A' $ and $ B'$ can be thought of as a pairing of $A\cup B$, i.e., a partition of $A\cup B$ into pairs $\{ a, \pi (a)\}$, $a\in A'$.
Note that each pairing $\pi$ of $A\cup B$ appears in a sum as above for exactly $2^{|A|}$ choices of $A'$. Indeed, for each of the $|A|$ pairs in $\pi$, we 
can choose one vertex that will belong to $A'$. Hence, by \eqref{eq:nicesign2} and by Lemma~\ref{lem:tposflow}, we have
\begin{align} \label{eq:kasteleyn}
\pf {K}^{A \cup B} =2^{-|A|} \sum_{A', B'}\det N^{A',B'}=   \frac{2^{-|A|}}{Z^{\emptyset}_{\aflow}} \sum_{A', B'} Z^{A',B'}_{\aflow}.
\end{align}

Recall that $\sets_{A',B'} \subset \sets_{A\cup B}$ is the image of $\flows_{A',B'}$ under the map $\setm$ from Section~\ref{sec:flows}.
We claim that each $\set \in \sets_{A\cup B}$ belongs to $\sets_{A',B'}$ for exactly $2^{k'(\set)}$ choices of $A'$, where $k'(\set)$ is the number of connected components of $\set$ 
containing a vertex in $A\cup B$. Indeed,
by Lemma~\ref{lem:flowprop} and~\ref{lem:revflowprop}, 
for each connected component $\ccomp$ of $\set$ that intersects $A\cup B$, there are exactly two ways of distributing the vertices in $V(\ccomp) \cap (A \cup B)$
between $A'$ and $B'$ (choosing one vertex to be connected to a source or a sink 
fixes the choices for all other vertices in $V(\ccomp) \cap (A \cup B)$ by the alternating property). 
Hence, by combining Theorem~\ref{thm:dcurrformula} and~\ref{thm:flowformula}, we have
\begin{align*}
\eqref{eq:kasteleyn} &= \frac{1}{Z^{\emptyset}_{\aflow}}\sum_{A', B'}  
 \sum_{\set \in \sets_{A', B'}}2^{-k'(\set)}   |\mathcal{E}_{\emptyset}(\set)| \prod_{e\in \set_1} \frac{x_e}{1-x^2_e}  \prod_{e\in \set_2} \frac{x^2_e}{1-x^2_e} \\
& =  \frac{1}{Z^{\emptyset}_{\aflow}}  \sum_{\set \in \sets_{A\cup B}}  2^{k'(\set)-k'(\set)}
  |\mathcal{E}_{\emptyset}(\set)| \prod_{e\in \set_1} \frac{x_e}{1-x^2_e}  \prod_{e\in \set_2} \frac{x^2_e}{1-x^2_e} \\
& =  \frac{Z^{A \cup B}_{\rcur} Z^{\emptyset}_{\rcur} }{Z^{\emptyset}_{\aflow}} \\ 
& =\frac{Z^{A \cup B}_{\rcur} }{Z^{\emptyset}_{\rcur}},
\end{align*}
where in the last equality we used Corollary~\ref{cor:empty} to get $Z^{\emptyset}_{\aflow}=(Z^{\emptyset}_{\rcur})^2$.
\end{proof}

We are now in a position to prove the rest of our main results.
\begin{proof}[Proof of Theorem~\ref{thm:pairing}]
Recall that $A$ and $B$ are placed around the outer face in such a way that all vertices of $A$ come before every vertex of $B$
in a counterclockwise order.
Let $\overline{ \para}_{A , B} \subseteq \sets_{A \cup B}$ be the image of the event
$\para_{A , B}$ under the map from Section~\ref{sec:currents}. 
Comparing the formulas in Theorem~\ref{thm:dcurrformula} and~\ref{thm:flowformula},
we have
\begin{align}
\IP_{\drcur}^{A\cup B} (\para_{A , B} )  &= \overline{ \IP}_{\drcur}^{A\cup B}  (\overline{ \para}_{A , B}) \nonumber \\&
 = \sum_{\set\in \overline{ \para}_{A , B}} \overline{ \IP}_{\drcur}^{A\cup B}  (\set)\nonumber  \\
&= \frac{Z^{A,B}_{\aflow}}{Z^{A\cup B}_{\rcur} Z^{\emptyset}_{\rcur}} \sum_{\set \in 
\overline{ \para}_{A , B}}  2^{k'(\set) -|A|}\overline{ \IP}_{\aflow}^{A, B}  (\set)\nonumber \\
& = \frac{ Z^{A,B}_{\aflow}}{Z^{A\cup B}_{\rcur} Z^{\emptyset}_{\rcur}} \sum_{\set \in\overline{ \para}_{A , B}} \overline{ \IP}_{\aflow}^{A, B}  (\set) \nonumber \\
& = \frac{ Z^{A,B}_{\aflow}}{Z^{A\cup B}_{\rcur} Z^{\emptyset}_{\rcur}} , \label{eq:parallel}
\end{align}
where the second to last equality holds true since, by the definition of $\para_{A , B}$, $k'(\set) = |A|$ for every $\set \in \overline{ \para}_{A , B}$. 
Moreover, since $\overline{ \IP}_{\aflow}^{A, B}$ is supported on $\sets_{A,B}$, to justify the last equality we have to show that $\overline{ \para}_{A , B} =\sets_{A,B} $.
To this end, note that by Lemma~\ref{lem:flowprop}, for each $F\in \flows_{A,B}$, we have that $\setm(F) \in \overline{ \para}_{A , B}$,
where $\setm$ denotes the map from Section~\ref{sec:flows}. 
Indeed, for contiguous sets $A$ and $B$, there is only one way of distributing the sources $A^{\source}$ and sinks $B^{\sink}$ 
between the connected components of an alternating flow in such a way that they alternate around the external face for each component. This means
that each connected component of $F$ connects a single point in $A$ to a single point in $B$.
On the other hand, by Lemma~\ref{lem:revflowprop}, for each $\set \in \overline{ \para}_{A , B}$,
there exists $F \in \flows_{A,B}$ such that $\setm(F) =\set$. Therefore, $\overline{ \para}_{A , B}=\sets_{A,B} $ and \eqref{eq:parallel} holds true.
Furthermore, using Lemma~\ref{lem:talaska} and~\ref{lem:gbk}, we get
\begin{align*} 
\frac{\det M^{A,B}  }{\pf {K}^{A \cup B}} = \frac{Z^{A,B}_{\aflow}Z^{\emptyset}_{\rcur}}{Z^{\emptyset}_{\aflow}Z^{A\cup B}_{\rcur} } = \frac{ Z^{A,B}_{\aflow}}{Z^{A\cup B}_{\rcur}Z^{\emptyset}_{\rcur}},
\end{align*}
where we used Corollary~\ref{cor:empty} to obtain that $Z^{\emptyset}_{\aflow}=(Z^{\emptyset}_{\rcur})^2$.
Together with \eqref{eq:parallel}, this finishes the proof.
\end{proof}

\begin{proof}[Proof of Theorem~\ref{thm:scalinglimit}]
Hongler in his PhD thesis~\cite{hongler} showed that in the setting of Theorem~\ref{thm:scalinglimit}, for any $a \in  A$ and $b \in B$, the following limit exists
\begin{align} \label{eq:hongler1}
f_{a,b}(D) :=\lim_{\epsilon \to 0}  \epsilon^{-1} \langle \sigma_{a_{\epsilon}} \sigma_{b_{\epsilon}}  \rangle,
\end{align}
where $a_{\epsilon}$ and $b_{\epsilon}$ are vertices on the outer face of $D_{\epsilon}$ approximating $a$ and $b$. Moreover, the limit is conformally covariant, i.e.,
for a domain $D'$ and a conformal equivalence $\varphi: D \to D'$, we have 
\begin{align} \label{eq:hongler2}
f_{a,b}(D) =   f_{\varphi(a),\varphi(b)}(D')  |\varphi'(a)\varphi'(b)|^{\frac12}.
\end{align}
Furthermore, for real numbers $a,b$, we have 
\[
f_{a,b}(\mathbb{H}) = \frac{2(\sqrt{2}+1)}{\pi |a-b| }.
\] 
The result now follows from Theorem~\ref{thm:pairing} since, by the expansions of the determinant and Pfaffian \eqref{eq:nicesign} and \eqref{eq:nicesign2},
the normalization factors from \eqref{eq:hongler1},
the constants ${2(\sqrt{2}+1)}/{\pi}$, and the terms containing derivatives of $\varphi$ from \eqref{eq:hongler2} cancel out.
\end{proof}

\noindent{\bf Acknowledgments.} The author thanks the anonymous referees for very useful suggestions.
This research was supported by the Knut and Alice Wallenberg foundation, and was conducted when the author was at Chalmers University of Technology and the University of Gothenburg.
\bibliographystyle{amsplain}
\bibliography{TotalPositivity1}
\end{document}